\documentclass{article}
\usepackage{amssymb}
\usepackage{amsmath}
\usepackage{amsthm}
\usepackage{mathabx}
\usepackage{enumerate}
\usepackage{xcolor}
\usepackage{changepage}
\usepackage{graphicx}

\usepackage{hyperref}

\theoremstyle{plain}
\everymath=\expandafter{\the\everymath\displaystyle}

\newtheorem{Thm}{Theorem}

\newcommand*{\Ann}{\ensuremath{\mathrm{Ann\,}}}
\newcommand*{\R}{\ensuremath{\mathbb{R}}}

\newcommand*{\C}{\ensuremath{\mathbb{C}}}

\newcommand*{\N}{\ensuremath{\mathbb{N}}}

\begin{document}

	\date{}

\author{
	L\'aszl\'o Sz\'ekelyhidi\\
	{\small\it Institute of Mathematics, University of Debrecen,}\\
	{\small\rm e-mail: \tt szekely@science.unideb.hu,}
}	
	\title{Spectral Synthesis on Continuous Images}

	\maketitle
	
	\begin{abstract}
	
	\end{abstract}
	
	\footnotetext[1]{The research was supported by the  the
		Hungarian National Foundation for Scientific Research (OTKA),
		Grant No.\ K-134191.}\footnotetext[3]{AMS (2000) Subject Classification: 43A45, 22D99}
	\bigskip\bigskip


	\section{Introduction}
Let $G$ be a locally compact Abelian group. Spectral synthesis deals with uniformly closed translation invariant linear spaces of continuous complex valued functions on $G$. Such a space is called a {\it variety}. We say that {\it spectral analysis} holds for a variety, if every nonzero subvariety contains a one dimensional subvariety. We say that a variety is {\it synthesizable}, if its finite dimensional subvarieties span a dense subspace in the variety. Finally, we say that {\it spectral synthesis} holds for a variety, if every subvariety of it is synthesizable. On commutative topological groups finite dimensional varieties of continuous functions are completely characterized: they are spanned by exponential monomials. {\it Exponential polynomials} on a topological Abelien group are defined as the elements of the complex algebra of continuous complex valued functions generated by all continuous homomorphisms into the multiplicative group of nonzero complex numbers ({\it exponentials}), and all continuous homomorphisms into the additive group of all complex numbers ({\it additive functions}).  An {\it exponential monomial} is a function of the form
$$
x\mapsto P\big(a_1(x),a_2(x),\dots,a_n(x)\big)m(x),
$$
where $P$ is a complex polynomial in $n$ variables, the $a_i$'s are additive functions, and $m$ is an exponential. If $m=1$, then we call it a {\it polynomial}. Every exponential polynomial is a linear combination of exponential monomials. For more about spectral analysis and synthesis on groups see \cite{MR2680008,MR3185617}.
\vskip.2cm

In \cite{MR2340978}, the authors characterized those discrete Abelian groups having spectral synthesis: spectral synthesis holds for every variety on the discrete Abelian group $G$, if and only if $G$ has finite torsion free rank. In particular, from this result it follows, that if spectral synthesis holds on $G$ and $H$, then it holds on $G\oplus H$. Unfortunately, such a result does not hold in the non-discrete case. Namely, by the fundamental result of L.~Schwartz \cite{MR0023948}, spectral synthesis holds on $\R$, but D.~I.~Gurevich showed in \cite{MR0390759} that spectral synthesis fails to hold on $\R\times\R$. A complete description of those locally compact Abelian groups where spectral synthesis holds for the space of all continuous functions was obtained in \cite{Sze23c}, where the present author proved the following two theorems:

\begin{Thm}
	Spectral synthesis holds on the compactly generated locally compact Abelian group $G$ if and only if it is topologically isomorphic to $\R^a\times Z^b\times C$, where $C$ is compact, and $a,b$ are nonnegative integers with $a\leq 1$.
\end{Thm}

\begin{Thm}
	Spectral synthesis holds on the locally compact Abelian group $G$ if and only if $G/B$ is topologically isomorphic to $\R^a\times Z^b\times F$, where $B$ is the subgroup of all compact elements in $G$, $F$ is a discrete Abelian group of finite rank, and $a,b$ are nonnegative integers with $a\leq 1$.
\end{Thm}

These characterization theorems describe those groups where all varieties are synthesizable -- another question is if spectral synthesis holds a given particular variety, even if it does not hold on the whole group. In \cite{Sze23}, we introduced the concept of localization, which is an effective tool to prove spectral synthesis on varieties.  In this paper we apply this method to show that any continuous image of a synthesizable variety is synthesizable as well.

\section{Preliminaries}
Here we summarize some known results we shall use in the subsequent paragraphs. 
\vskip.2cm

Given a locally compact Abelian group $G$ the continuous complex homomorphisms of $G$ into the multiplicative group of nonzero complex numbers, resp. into the additive group of complex numbers are called {\it exponentials}, resp. {\it additive functions}. A product of additive functions is called a {\it monomial}, and a linear combination of monomials is called a {\it polynomial}. A product of an exponential and a polynomial is called an {\it exponential monomial}, and if the exponential is $m$, then we call the exponential monomial an {\it $m$-exponential monomial}. Hence polynomials are exactly the $1$-exponential monomials. Linear combinations of exponential monomials are called {\it exponential polynomials}.  One dimensional varieties are exactly those spanned by an exponential, and finite dimensional varieties are exactly those spanned by exponential monomials (see \cite{MR3185617}). The {\it variety of the function $f$} in $\mathcal C(G)$, denoted by $\tau(f)$, is the intersection of all varieties including $f$.
\vskip.2cm

It is known that the dual space of $\mathcal C(G)$ can be identified with the space $\mathcal M_c(G)$ of all compactly supported complex Borel measures on $G$. This space is called the {\it measure algebra} of $G$ -- it is a topological algebra with the linear operations, with the convolution of measures and with the weak*-topology arising from $\mathcal C(G)$. On the other hand, the space $\mathcal C(G)$ is a topological vector module over the measure algebra under the action realized by the convolution of measures and functions. The annihilators of subsets in $\mathcal C(G)$ and the annihilators of subsets in $\mathcal M_c(G)$ play an important role in our investigation. For each closed ideal $I$ in $\mathcal M_c(G)$ and for every variety $V$ in $\mathcal C(G)$, $\Ann I$ is a variety in $\mathcal C(G)$ and $\Ann V$ is a closed ideal in  $\mathcal M_c(G)$. Further, we have
$$
\Ann \Ann I=I\enskip\text{and}\enskip \Ann \Ann V=V
$$
(see \cite[Section 11.6]{MR3185617}, \cite[Section 1]{MR3502634}).
\vskip.2cm

The Fourier--Laplace transformation (shortly: Fourier transformation) on the measure algebra is defined as follows: for every exponential $m$ on $G$ and for each measure $\mu$ in $\mathcal M_c(G)$ its {\it Fourier transform} is 
$$
\widehat{\mu}(m)=\int \widecheck{m}\,d\mu,
$$
where $\widecheck{m}(x)=m(-x)$ for each $x$ in $G$. The Fourier transform $\widehat{\mu}$ is a complex valued function defined on the set of all exponentials on $G$. As the mapping $\mu\mapsto \widehat{\mu}$ is linear and $(\mu*\nu)\,\widehat{}=\widehat{\mu}\cdot \widehat{\nu}$, all Fourier transforms form a function algebra. By the injectivity of the Fourier transform, this algebra is isomorphic to $\mathcal M_c(G)$. If we equip the algebra of Fourier transforms by the topology arising from the topology of $\mathcal M_c(G)$, then we get the {\it Fourier algebra} of $G$, denoted by $\mathcal A(G)$. In fact, $\mathcal A(G)$ can be identified with $\mathcal M_c(G)$. We utilize this identification:  for instance, every ideal in $\mathcal A(G)$ is of the form $\widehat{I}$, where $I$ is an ideal in $\mathcal M_c(G)$. Based on this fact, we say that {\it spectral synthesis holds for the ideal $\widehat{I}$ in $\mathcal A(G)$}, if spectral synthesis holds for $\Ann I$ in $\mathcal C(G)$.
\vskip.2cm

We shall use the polynomial derivations on the Fourier algebra. A {\it derivation} on $\mathcal A(G)$ is a linear operator $D:\mathcal A(G)\to \mathcal A(G)$ such that
$$
D(\widehat{\mu}\cdot \widehat{\nu})=D(\widehat{\mu})\cdot \widehat{\nu}+\widehat{\mu}\cdot D(\widehat{\nu})
$$
holds for each $\widehat{\mu}, \widehat{\nu}$. We say that $D$ is a {\it first order derivation}. Higher order derivations are defined inductively: for a positive integer $n$ we say that linear operator $D$ on $\mathcal A(G)$ is a {\it derivation of order $n+1$}, if the two variable operator
$$
(\widehat{\mu}, \widehat{\nu})\mapsto D(\widehat{\mu}\cdot \widehat{\nu})-D(\widehat{\mu})\cdot \widehat{\nu}-\widehat{\mu}\cdot D(\widehat{\nu})
$$ 
is a derivation of order $n$ in both variables. The identity operator $id$ is considered a derivation of order $0$. All derivations form an algebra of operators, and the derivations in subalgebra generated by all first order derivations are called  {\it polynomial derivations}. They have the form $P(D_1,D_2,\dots,D_k)$, where $D_1,D_2,\dots,D_k$ are first order derivations, and $P$ is a complex polynomial in $k$ variables. In \cite{Sze23}, we proved the following result:

\begin{Thm}\label{deri}
	The linear operator $D$ on $\mathcal A(G)$ is a polynomial derivation if and only if there exists a unique polynomial $f_D$ such that
	$$
	D\widehat{\mu}(m)=\int f_D(x)m(-x)\,d\mu(x)
	$$
	holds for each $\widehat{\mu}$ in $\mathcal A(G)$ and for every exponential $m$ on $G$. 
\end{Thm}

Iin \cite{Sze23}, we introduced the following concepts. Given an ideal $\widehat{I}$ in $\mathcal A(G)$ and an exponential $m$, we denote by $\mathcal P_{\widehat{I},m}$ the family of all polynomial derivations $P(D_1,D_2,\dots,D_k)$ which annihilate $\widehat{I}$ at $m$. This means that 
$$
\partial^{\alpha}P(D_1,D_2,\dots,D_k)\widehat{\mu}(m)=0
$$
for each multi-index $\alpha$ in $\N^k$, for every exponential $m$, and for every $\widehat{\mu}$ in $\widehat{I}$. The dual concept is the following: given a family $\mathcal P$ of polynomial derivations and an exponential $m$ we denote by $\widehat{I}_{\mathcal P,m}$ the set of all functions $\widehat{\mu}$ which are annihilated by every derivation in the family $\mathcal P$ at $m$. Then $\widehat{I}_{\mathcal P,m}$ is a closed ideal. Obviously, 
$$
\widehat{I}\subseteq \bigcap_m \widehat{I}_{\mathcal P_{\widehat{I},m},m}
$$
holds for every ideal $\widehat{I}$. We call $\widehat{I}$ {\it localizable}, if we have equality in this inclusion. In other words, the ideal $\widehat{I}$ in $\mathcal A(G)$ is localizable if and only if it has the following property: if $\widehat{\mu}$ is annihilated by all polynomial derivations, which annihilate $\widehat{I}$ at each $m$, then $\widehat{\mu}$ is in $\widehat{I}$. The main result in \cite{Sze23} is the following:

\begin{Thm}\label{loc}
	Let $G$ be a locally compact Abelian group. The ideal $\widehat{I}$ in the Fourier algebra is localizable if and only if $\Ann I$ is synthesizable.
\end{Thm}

\section{Main result}

Let $G$ be a locally compact Abelian group. Given a variety $V$ in $\mathcal C(G)$ a {\it continuous image of $V$} is a variety $W$ on a locally compact Abelian group $H$ such that there exists a continuous surjective homomorphism $\Phi:G\to H$ such that the function $\varphi$ is in $W$ if and only if the function $\varphi\circ \Phi$ is in $V$.

\begin{Thm}\label{mainthm}
	Let $G$ be a locally compact Abelian group and $V$ a variety on $G$. If $V$ is synthesizable, then every continuous image of $V$ is synthesizable. 
\end{Thm}

\begin{proof}
	In the light of Theorem \ref{loc}, it is enough to show that if $V$ is synthesizable, then, for every continuous image $W$ of $V$, the ideal $(\Ann W)\,\widehat{}$\, is localizable.
	\vskip.2cm
	
	Assume that $W$ is a variety on the locally compact Abelian group $H$, and $\Phi:G\to H$ is a continuous surjective homomorphism such that the function $\varphi$ is in $W$ if and only if the function $\varphi\circ \Phi$ is in $V$. We denote $\Ann V$, resp. $\Ann W$ by $I$, resp. $J$.
	\vskip.2cm
	
	First we observe that for every exponential $m$ on $H$, the function $m\circ \Phi$ is an exponential on $G$. Similarly, for every additive function $a$ on $H$, the function $a\circ \Phi$ is an additive function on $G$. From this we conclude that for every polynomial $p$, resp. $m$-exponential monomial $\varphi$ on $H$, the function $p\circ \Phi$, resp. the function $\varphi\circ \Phi$ is a polynomial, resp. an $m\circ \Phi$-exponential monomial on $G$. 
	\vskip.2cm
	
	The mapping $\Phi$ induces a continuous algebra homomorphism $\Phi_H$ of the measure algebra $\mathcal M_c(G)$ into the measure algebra $\mathcal M_c(H)$ is the following manner: for each measure $\mu$ on $G$ we let 
	$$
	\langle \Phi_H(\mu),\varphi\rangle=\langle \mu,\varphi\circ \Phi\rangle
	$$
	whenever $\varphi$ is in $\mathcal C(H)$. It is easy to se that $\Phi_H$ is a continuous linear functional on $\mathcal C(H)$, hence it is in $\mathcal M_c(H)$. We can check easily that $\Phi_H$ is a continuous algebra homomorphism. 
	\vskip.2cm
	
	In fact, $\Phi_H$ is surjective. Indeed, for each $u$ in $H$ there is an $x$ in $G$ such that $u=\Phi(x)$. It follows, for each $\varphi$ in $\mathcal C(H)$,
	$$
	\langle \Phi_H(\delta_x),\varphi\rangle=\langle \delta_x,\varphi\circ \Phi\rangle=\varphi(\Phi(x))=\varphi(u)=\langle \delta_u,\varphi\rangle,
	$$
	hence $\Phi_H(\delta_x)=\delta_{\Phi(x)}$. As each measure in $\mathcal M_c(H)$ is a weak*-limit of finitely supported measures, and all finitely supported measures are in the image of $\Phi_H$, we conclude that $\Phi_H$ is surjective.
	\vskip.2cm
	
	The adjoint mapping of $\Phi_H$ is a linear mapping from $\mathcal M_c(H)^*$ onto $\mathcal M_c(G)^*$. As these spaces are identified by $\mathcal C(H)$, resp. $\mathcal C(G)$, we realize the adjoint of $\Phi_H$ as the mapping
	$$
	\Phi_H^*(\varphi)=\varphi\circ \Phi
	$$
	for each $\varphi$ in $\mathcal C(H)$. As $\Phi_H$ is surjective, so is $\Phi_H^*$, and we infer that every function in $\mathcal C(G)$ is of the form $\varphi\circ \Phi$ with some $\varphi$ in $\mathcal C(H)$.
	\vskip.2cm
	
	Obviously, $\Phi_H$ induces a continuous algebra homomorphism from the Fourier algebra $\mathcal A(G)$ onto the Fourier algebra $\mathcal A(H)$, which we denote by $\widehat{\Phi}_H$, satisfying
	$$
	\widehat{\Phi}_H(\widehat{\mu})= \Phi_H(\mu)\,\widehat{}
	$$
	for each $\mu$ in $\mathcal M_c(G)$. We claim that the ideal $\widehat{I}$ is mapped onto $\widehat{J}$ by  $\widehat{\Phi}_H$. Let $\widehat{\mu}$ be in $\widehat{I}$, then $\mu*f=0$ for each $f$ in $V$. We need to show that $\widehat{\Phi}_H(\widehat{\mu})$ is in $\widehat{J}$, that is,  $\Phi_H(\mu)$ annihilates $W$. If $\varphi$ is in $W$, then $f=\varphi\circ \Phi$ is in $V$, hence
	$$
	\Phi_H(\mu)*\varphi(u)=\int_H \varphi(u-v)\,d\Phi_H(\mu)(v)=\int_G \varphi(\Phi(x)-\Phi(y)\,d\mu(y)=
	$$
	$$
	\int_G \varphi(\Phi(x-y))\,d\mu(y)=\int_G (\varphi\circ \Phi)(x-y)\,d\mu(y)=
	$$
	$$
	 \int_G f(x-y)\,d\mu(y)=\mu*f(x)=0.
	$$
	On the other hand, if $\widehat{\nu}$ is in $\widehat{J}$, then $\nu$ is in $J=\Ann W$, further $\nu=\Phi_H(\mu)$ for some $\mu$ in $\mathcal M_c(G)$. We want to show that $\mu$ is in $\Ann V=\Ann \Ann I$. Assuming the contrary, there exists an $f$ in $V$ such that $\mu*f\ne 0$. We have $f=\varphi\circ \Phi$ for some $\varphi$ in $W$, and this implies
	$$
	\nu*\varphi=\Phi_H(\mu)*\varphi=\mu*(\varphi\circ \Phi)=\mu*f\ne 0,
	$$
	a contradiction, as $\nu$ is in $\Ann W$ and $\varphi$ is in $W$. We conclude that $\mu$ is in $\Ann V$, hence the mapping $\widehat{\Phi}_H:\widehat{I}\to \widehat{J}$ is onto.
	\vskip.2cm
	
	Now we are ready to show that $\widehat{J}$ is localizable, if $V$ is synthesizable, i.e. $\widehat{I}$ is localizable. Let $\widehat{\nu}$ be  in $\widehat{J}_{\mathcal P_{m, \widehat{J}},m}$ -- we need to show that $\widehat{\nu}$ is in $\widehat{J}$. Here $m$ is an arbitrary exponential on $H$, hence $m\circ \Phi$ is an exponential on $G$. Let $\widehat{\nu}=\widehat{\Phi}_H(\widehat{\mu})$, where $\widehat{\mu}$ is in $\mathcal A(G)$. It is enough to show that $\widehat{\mu}$ is in $\widehat{I}$. Using the localizability of $\widehat{I}$, it is enough to show that every derivation in $\widehat{I}_{\mathcal P_{m\circ \Phi, \widehat{I}},m\circ \Phi}$ annihilates $\widehat{\mu}$ at $m\circ \Phi$. Let $D$ be a polynomial derivation in $\widehat{I}_{\mathcal P_{m\circ \Phi, \widehat{I}},m\circ \Phi}$. It has the form 
	$$
	D\widehat{\mu}(m\circ \Phi)=\int_G f_D(x) (m\circ \Phi)(-x)\,d\mu(x)
	$$
	for each exponential $m$ on $H$, where $f_D:G\to\C$ is a polynomial. We have seen above that $f_D$ can be written as $f_D=p_D\circ \Phi$ with some polynomial $p_D:H\to\C$. Hence we have
	$$
	D\widehat{\mu}(m\circ \Phi)=\int_G f_D(x) (m\circ \Phi)(-x)\,d\mu(x)=\int_G (p_D\circ \Phi)(x) (m\circ \Phi)(-x)\,d\mu(x),
	$$
	or 
	$$
	D\widehat{\nu}(m)=D\widehat{\Phi}_H(\widehat{\mu})(m)=D\Phi_H(\mu)\,\widehat{}(m)=\int_G p_D(u) m(-u)\,d\Phi_H(\mu)(u).
	$$
	This means that $D$ induces a polynomial derivation on $\mathcal A(H)$, which is in $\widehat{J}_{\mathcal P_{m, \widehat{J}},m}$. By assumption, this derivation annihilates $\widehat{\nu}$ at $m$, which implies that $D$ annihilates $\widehat{\mu}$ at $m\circ \Phi$. As this holds for each $D$ in $\widehat{I}_{\mathcal P_{m\circ \Phi, \widehat{I}},m\circ \Phi}$, by the localizability of $\widehat{I}$, we conclude that $\widehat{\mu}$ is in $\widehat{I}$, thus $\widehat{\nu}$ is in $\widehat{J}$, and our theorem is proved.
\end{proof}

\end{document}